\documentclass[11pt]{amsart}
\usepackage[dvipsnames]{xcolor}
\usepackage[utf8]{inputenc}
\usepackage{amssymb,amsmath,amsthm,amsfonts,fullpage,float,latexsym,bbm,microtype,cite,fancyvrb, cite, verbatim, parskip}
\usepackage{tikz}
\usepackage{enumerate}
\usetikzlibrary{decorations.pathreplacing}
\usepackage{hyperref}
\hypersetup{colorlinks=true, linkcolor=blue, citecolor=magenta, filecolor=magenta, urlcolor=magenta}
\usetikzlibrary{graphs,graphs.standard,calc}
\usepackage{subcaption}
\usepackage[colorinlistoftodos]{todonotes}
\definecolor{andresblue}{rgb}{0,0.72,0.92}
\definecolor{andrespink}{rgb}{1,0,1}

\definecolor{munsell}{rgb}{0.0, 0.5, 0.69}
\newcommand{\andres}[1]{\todo[size=\tiny,inline,color=munsell!50]{#1 \\ \hfill --- Andr\'es}}

\definecolor{suprise}{rgb}{1, 0.0, 0.69}
\newcommand{\ben}[1]{\todo[size=\tiny,inline,color=suprise!50]{#1 \\ \hfill --- Ben}}

\definecolor{kojigreen}{rgb}{0, 1, 0}

\newcommand{\kk}{\Bbbk}

\makeatletter
\newtheorem*{rep@theorem}{\rep@title}\newcommand{\newreptheorem}[2]{%
\newenvironment{rep#1}[1]{%
\def\rep@title{\bf #2 \ref{##1}}%
\begin{rep@theorem}}%
{\end{rep@theorem}}}
\makeatother
\newreptheorem{theorem}{Theorem}

\newtheorem{theorem}{Theorem}[section]
\newtheorem{proposition}[theorem]{Proposition}

\newtheorem{conjecture}[theorem]{Conjecture}
\newtheorem{lemma}[theorem]{Lemma}
\newtheorem{corollary}[theorem]{Corollary}
\theoremstyle{definition}
\newtheorem{remark}[theorem]{Remark}
\newtheorem{definition}[theorem]{Definition}

\newtheorem{example}[theorem]{Example}

\newtheorem{question}[theorem]{Question}

\renewcommand{\emptyset}{\varnothing}

\DeclareMathOperator{\conv}{conv}

\newcommand{\R}{\mathbb{R}}
\newcommand{\Z}{\mathbb{Z}}

\newcommand{\bC}{\mathbb C}

\newcommand{\ZZ}{\mathbb{Z}}

\begin{document}

\title{Matching polytopes, Gorensteinness, and the integer decomposition property}

\author{Benjamin Eisley}
\address{\scriptsize{Department of Mathematics, UC Berkeley}}
\email{\scriptsize{beisley2025@berkeley.edu}}

\author{Koji Matsushita}
\address{\scriptsize{Department of Pure and Applied Mathematics, Graduate School of Information Science and Technology, Osaka University}}
\email{\scriptsize{k-matsushita@ist.osaka-u.ac.jp}}

\author{Andr\'es R. Vindas-Mel\'endez}
\address{\scriptsize{Department of Mathematics, Harvey Mudd College}, \url{https://math.hmc.edu/arvm}}
\email{\scriptsize{avindasmelendez@g.hmc.edu}}


\begin{abstract}
The matching polytope of a graph $G$ is the convex hull of the indicator vectors of the matchings on $G$.
We characterize the graphs whose associated matching polytopes are Gorenstein, and then prove that all Gorenstein matching polytopes possess the integer decomposition property.
As a special case study, we examine the matching polytopes of wheel graphs and show that they are \emph{not} Gorenstein, but \emph{do} possess the integer decomposition property.
\end{abstract}

\maketitle

\section{Introduction}
A lattice polytope $P$ (i.e., a convex polytope with vertices in $\Z^d$) is \emph{Gorenstein of index} $k \in \Z_{>0}$ if $kP$ is a reflexive polytope (i.e., its dual is also a lattice polytope). 
Gorenstein polytopes have long enticed mathematicians working at the intersection of combinatorics, discrete geometry, and commutative algebra, as they can be studied through the lenses of each of the aforementioned areas.
For example, one might deviate from the geometric definition of Gorenstein above, and instead adopt a commutative-algebraic perspective by studying this property through the Ehrhart ring of $P$. 
The Ehrhart ring of $P$ is the monoid algebra (a special type of graded ring) generated by the monoid of lattice points in the cone over $P$ \cite{BrunsGubeladze}; this a special toric ring.
The Ehrhart ring is Gorenstein if it is a is a Cohen-Macaulay ring that has a dualizing module isomorphic to a shift of itself (see \cite{BrunsHerzog,BrunsGubeladze} for more background).
We say that $P$ is Gorenstein if its Ehrhart ring is Gorenstein. 

One motivation for our work comes from Ehrhart theory, i.e., the study of lattice points in polytopes.
For a lattice polytope $P$, we define the \emph{Ehrhart series} of $P$ to be the generating function 
\[
  \operatorname{Ehr}(P;z):=\sum_{t\in  {\Z_>0}}|tP \cap \Z^n| \, z^t=\frac{h^*(P;z)}{(1-z)^{\operatorname{dim}(P)+1}} \, .
\]
The polynomial $h^*(P;z)=\sum_{i=0}^n h^*_iz^i$ is known as the \emph{$h^*$-polynomial} of~$P$ and is of degree less than $n+1$ and has nonnegative integer coefficients. 
For more on Ehrhart theoretic background, we recommend the following references: \cite{BeckRobins, BeckSanyal, Braun, Ehrhart}.

A lattice polytope $P$ is said to exhibit the \emph{integer decomposition property} (\emph{IDP}, for short) if for any $t \in \Z_{>0}$ and $\alpha \in tP \cap \Z^d$, there exist $\alpha_1, \ldots, \alpha_t \in P \cap \Z^d$ with $\alpha = \alpha_1 + \cdots + \alpha_t$.    
In \cite{OhsugiHibi}, Hibi an Ohsugi conjectured the following result which was proven by Adiprasito, Papadakis, Petrotou, and Steinmeyer in \cite{AdiprasitoPapadakisPetrotouSteinmeyer}:
\begin{theorem}[\cite{AdiprasitoPapadakisPetrotouSteinmeyer}]\label{AdiprasitoPapadakisPetrotouSteinmeyer}
        If $P$ is Gorenstein and has the IDP, then $h^*(P;z)$ is unimodal.     
\end{theorem}

Relaxing the conjecture (now theorem) slightly, Schepers and Van Langenhoven in \cite{SchepersVanLangenhoven}, asked the following:

\begin{question}
    If $P$ has the IDP, is it true that $h^*(P;z)$ is unimodal? 
\end{question}
\noindent As far as the authors know, this question still remains unanswered in full generality.

There are many well-investigated classes of toric rings, including: 
\begin{itemize}
\item \emph{Hibi rings}, i.e. the toric rings of order polytopes (see, for example, \cite{DasMukherjee, Ene, EneHerzogSaeediMadani, Hibi, HigashitaniMatsushita, HigashitaniNakajima, Miyazaki}).

\item Edge rings, i.e. the toric rings of edge polytopes (see, for example, \cite{HamanoHibiOhsugi, HigashitaniMatsushita, OhsugiHibi2, SimisVasconcelosVillarreal}).

\item Stable set rings, i.e. the toric rings of stable-set polytopes: of those polytopes whose vertices are the characteristic vectors of stable (independent) sets of a graph (see, for example, \cite{chvatal1975stable, HigashitaniMatsushita}). 

\end{itemize}

In this work, we study a special class of stable set polytopes known as \emph{matching polytopes}, whose vertices are the indicator vectors of  the matchings on a graph.
The outline and main results of this paper are as follows:
\begin{itemize}
    \item In Section \ref{sec:prelims}, we introduce preliminaries and background on lattice polytopes, Ehrhart rings, and graph theory.

    \item In Section \ref{sec:gorenstein}, we characterize the graphs whose associated matching polytopes are Gorenstein; this is the subject of Theorem \ref{thm:gor}.
    We then show Gorenstein matching polytopes have the IDP.

    \item Finally, in Section \ref{sec:wheel_graphs_IDP}, we study the matching polytopes associated to wheel graphs.
    We use tools developed in Section \ref{sec:gorenstein} to show that these matching polytopes are \emph{not} Gorenstein, but \emph{do} have the IDP.
    We conclude by presenting some data on the $h^*$-vector of the matching polytopes of wheel graphs, which appears to unimodal.
\end{itemize}


\section{Preliminaries}\label{sec:prelims}

\subsection{Ehrhart rings}

A \emph{lattice polytope} $P \subset \R^d$ is the convex hull of finitely many integer lattice points in $\Z^d$.

\begin{definition}
    The \emph{Ehrhart ring} $A(P)$ of $P$ is the finitely generated graded $k$-algebra:
    \begin{align*}
        &A(P):=\kk[z_1^{a_1}\cdots z_d^{a_d} z_{d+1}^{\eta} \mid \eta \in \Z_{>0}, \; (a_1,\ldots,a_d) \in \eta P\cap \Z^d], \\
        &\deg(z_1^{a_1}\cdots z_d^{a_d} z_{d+1}^{\eta}):=\eta.
    \end{align*}
\end{definition}

\begin{example}\label{kneqA}
    Let $P=\conv(\{(0,0,0),(1,1,0),(1,0,1),(0,1,1)\})$.
    Then:
    \begin{equation*}
        A(P)=\kk[\; z_4, \;z_1 z_2 z_4,\; z_1 z_3 z_4,\; z_2 z_3 z_4, \; z_1 z_2 z_3 z_4^2].
    \end{equation*}
    The monomial $z_1 z_2 z_3 z_4^2$ corresponds to the lattice point $(1,1,1)\in (2P\cap \Z^3)$.
    Since $z_1 z_2 z_3 z_4^2$ is independent of the other generators, in this example, $A(P)$ is not standard graded.
\end{example}

The Ehrhart ring $A(P)$ is standard graded if and only if $P$ has the \emph{integer decomposition property}:
\begin{definition}
    A lattice polytope $P$ has the \emph{integer decomposition property} (\emph{IDP}, for short) if for any $t \in \Z_{>0}$ and $\alpha \in tP \cap \Z^d$, there exist $\alpha_1, \ldots, \alpha_t \in P \cap \Z^d$ with $\alpha = \alpha_1 + \cdots + \alpha_t$.    
\end{definition}

Similarly, the Ehrhart ring $A(P)$ is Gorenstein if and only if some translated dilate of $P$ is \emph{reflexive}.
\begin{definition}
    Let $P \subset \R^d$ be a polytope.
    Its \emph{(polar) dual} is:
    \begin{equation*}
        P^* := \left\{y \in \R^d : \langle x, y\rangle \ge -1 \text{ for all } x \in P \right\}.
    \end{equation*}
    We call $P$ \emph{reflexive} if both $P$ and $P^*$ are lattice polytopes.
\end{definition}

\begin{theorem}[{\cite[Theorem~1.1]{DeNegriHibi}}]\label{refl}
    The following are equivalent:
    \begin{itemize}
        \item[(i)] The Ehrhart ring $A(P)$ of a lattice polytope $P$ is Gorenstein.
        \item[(ii)] There exist $k \in \Z_{>0}$ and $\alpha \in \Z^d$ such that $k P - \alpha$ is reflexive.
        \item[(iii)] There exist $k \in \Z_{>0}$ and $\alpha \in k P \cap \Z^d$ such that $\alpha$ has height $1$ over each facet of $P$.
    \end{itemize}
\end{theorem}
    
A polytope $P$ is \emph{Gorenstein of index} $k$ if $k P +\alpha$ is a reflexive polytope for some $\alpha \in \Z^d$.


\subsection{Graph theory} 

Let $G$ be a graph.
We call the vertex set $V(G)$ and the edge set $E(G)$.
For $U \subseteq V(G)$, we use $G[U]$ to denote the subgraph of $G$ induced by $U$.
Let $u \in V(G)$.
We adopt the following shorthand:
\begin{enumerate}
    \item $|G| := |V(G)|$,
    \item $V := V(G)$, $E: = E(G)$,
    \item $V(U): = V(G[W])$, $E(U): = E(G[U])$,
    \item $G \setminus u := G[V(G) \setminus \{ u \}]$.
\end{enumerate}

We say $G$ is \emph{disconnected} if $V$ can be partitioned into $U$ and $W$ such that no edges of $G$ link the two partitions.
Otherwise, we say $G$ is \emph{connected}.
If $G$ is connected, but $G \setminus u$ is disconnected for some $u \in V$, we call $u$ a \emph{cut vertex} of $G$.
We say that $G$ is \emph{$2$-connected} if $G$ has no cut vertices.
A \emph{block} of $G$ is a maximal $2$-connected component.
The following characterizes $2$-connected graphs.

\begin{theorem}[{ \cite[Proposition~3.1.1]{Diestel}}]\label{thm_factor}
\;
A graph $G$ is $2$-connected if and only if $G$ has an open ear decomposition, i.e., $G$ can be decomposed as \[C\cup \mathtt{P}_1\cup\cdots\cup \mathtt{P}_r,\] where $C$ is a cycle, $\mathtt{P}_i$ is a path and \[(C\cup \mathtt{P}_1 \cup \cdots \cup \mathtt{P}_{i-1})\cap \mathtt{P}_i\] consists of just two vertices for each $i$.
\end{theorem}

We say that a graph $G$ is \emph{almost bipartite} if there exists a vertex $u \in G$ such that $G \setminus u$ is bipartite.

Some graphs that are of relevance to this work include: 
\begin{itemize}
    \item the path on $n$ vertices $\mathtt{P}(n)$. Let $u_{\mathtt{P}}$ and $w_{\mathtt{P}}$ denote the endpoints of $\mathtt{P}(n)$.
    Then the \emph{interior} of $\mathtt{P}(n)$ is the subgraph of $\mathtt{P}(n)$ induced by $V(\mathtt{P}(n)) \setminus \{ u_{\mathtt{P}}, w_{\mathtt{P}} \}$.

    \item the cycle on $n$ vertices $C_n$.
    We say that a cycle $C$ is \emph{even} if $|C|$ is even, and \emph{odd} if $|C|$ is odd.
    
    \item The \emph{chortling cycle} $C'_5$ on $5$ vertices is a cycle with two chords (see Figure \ref{fig:chortling}):
    \begin{align*}
        V(C'_5) &= \ZZ_5, \\
        E(C'_5) &= \left\{ \{i, i+1\} : i \in \ZZ_5  \right\} \sqcup \{ \{1, 3\}, \{2, 4\} \}.
    \end{align*}

    \item The \emph{wheel graph} $W_n$ on $n \geq 4$ vertices consists of a cycle of length $n-1$ and a central vertex connected to each vertex in the cycle by an edge called a \emph{spoke} (see Figure \ref{fig:W4}). 
    In particular:
    \begin{align*}
        V(W_n) &= \ZZ_{n-1} \sqcup \{ v \}, \\
        E(W_n) &= \left\{ \{i, i+1\}, \{i, v\} : i \in \ZZ_{n-1} \right\}.
    \end{align*}

    \item The \emph{complete multipartite graph} $K_{r_1, ..., r_n}$ ($n > 1$) is any graph on a vertex set of the form $\bigsqcup_{i=1}^n U_i$ ($|U_i| = r_i$ for $i = 1, ..., n$) with edge set:
    \begin{equation*}
        E\left(K_{r_1, ..., r_n}\right) := \left\{\{u,u'\} \colon u \in U_i, u' \in U_j, 1 \leq i < j \leq n\right\}
    \end{equation*}
    (see Figure \ref{fig:multipartite}).
    When $r_1=\cdots=r_n=1$, we denote $K_{\underbrace{1,\ldots,1}_n}$ by $K_n$, and call this the \emph{complete graph} on $n$ vertices.
\end{itemize}

\begin{figure}[H]
    \centering
    \begin{minipage}{0.32\textwidth}
        \centering
        \begin{tikzpicture}
            \foreach \i in {1,...,5} {
                \node[circle, draw, fill=white, minimum size=5pt, inner sep=0pt] (\i) at ({90+72*(\i-1)}:1) {};
            }
            \foreach \i in {1,...,5} {
                \pgfmathtruncatemacro{\j}{mod(\i,5) + 1}
                \draw (\i) -- (\j);
            }
            \draw (3) -- (5);
            \draw (2) -- (4);
        \end{tikzpicture}
        \caption{The chortling  cycle $C'_5$.}\label{fig:chortling}
    \end{minipage}
    \begin{minipage}{0.32\textwidth}
        \centering
        \begin{tikzpicture}
            \foreach \i in {1,...,3} {
                \node[circle, draw, fill=white, minimum size=5pt, inner sep=0pt] (\i) at ({90+120*(\i-1)}:1) {};
            }
            \node[circle, draw, fill=white, minimum size=5pt, inner sep=0pt] (4) at (0,0) {};
            \foreach \i in {1,...,4} {
                \pgfmathtruncatemacro{\j}{mod(\i,3) + 1}
                \draw (\i) -- (\j);
                \draw (\i) -- (4);
            }
        \end{tikzpicture}
        \caption{The wheel graph $W_4$.}\label{fig:W4}
    \end{minipage}
    \begin{minipage}{0.32\textwidth}
        \centering
        \begin{tikzpicture}
            \foreach \i in {1,...,4} {
                \node[circle, draw, fill=white, minimum size=5pt, inner sep=0pt] (\i) at ({90*(\i-1)}:1) {};
            }
            \foreach \i in {1,...,4} {
                \pgfmathtruncatemacro{\j}{mod(\i,4) + 1}
                \draw (\i) -- (\j);
            }
                \draw (2) -- (4);
        \end{tikzpicture}
        \caption{The graph $K_{1,1,2}$.}\label{fig:multipartite}
    \end{minipage}
\end{figure}


\subsection{Matching polytopes}

Let $m \subseteq E(G)$ such that every vertex of $G$ is incident to at most one edge in $m$.
Then we call $m$ a \emph{matching} on $G$, and identify it with the indicator vector $m \in \R^E$ given by:
\begin{equation*}
    m(e) :=
    \begin{cases}
        1 & \textrm{ if } e \in m, \\
        0 & \textrm{ if else}.
    \end{cases}
\end{equation*}
We say a vertex $v$ is \emph{saturated} by a matching $m$ if some edge in $m$ is incident to $v$.
We say a matching $m$ is \emph{perfect} if it saturates every $u \in V(G)$.

The \emph{matching polytope} associated to $G$ is the convex hull of all matchings on $G$:
\[
    P_M(G) := \conv\left\{ m \in \R^E \mid m \textrm{ is a matching on } G \right\}.
\]

\begin{example}
    Consider $W_4$, the wheel graph on $4$ vertices.
    Let $0$, $1$, and $2$ denote the vertices on the outer cycle of $W_4$, and let $v$ denote the central vertex.
    Then $W_4$ has $9$ matchings:
    \begin{itemize}
        \item Three matchings consisting only of an inner spoke: $\{ \{ v, 0\} \}$, $\{ \{v, 1\} \}$, and $\{ \{v, 2\} \}$,
        \item Three matchings consisting only of an outer edge: $\{ \{ 1, 2 \} \}$, $\{ \{ 2, 0\} \}$, and $\{ \{ 0, 1 \} \}$,
        \item Three perfect matchings consisting of an inner spoke paired with an outer edge: $\{ \{v, 0\}, \{1, 2\} \}$, $\{ \{v, 1\}, \{2, 0\} \}$, and $\{ \{v, 2\}, \{0, 1\}\}$.
    \end{itemize}
    For concreteness, let us identify $\R^E$ with $\R^{|E|}$ by ordering the edges: $\{ v, 0\}, \{v, 1\}, \{v, 2\}, \{ 1, 2 \}, \{ 2, 0\}, \{ 0, 1 \}$.
    Then the matching polytope of $W_4$ is given by:
    \begin{equation*}
        P_M(W_4) =
        \conv\left\{ 
        e_1, e_2, e_3, e_4, e_5, e_6, e_1 + e_4, e_2 + e_5, e_3 + e_6
        \right\},
    \end{equation*}
    where $e_i$ denotes the $i^{\textrm{th}}$ standard basis vector in $\R^6$.

\end{example}

Now, let $s \subseteq V(G)$ such that each edge of $G$ is incident to at most one vertex in $s$. Then we call $s$ a \emph{stable set} on $G$, and identify it with the indicator vector $s \in \R^V$ given by:
\begin{equation*}
    s(v) :=
    \begin{cases}
        1 & \textrm{ if } v \in s, \\
        0 & \textrm{ if else.}
    \end{cases}
\end{equation*}
The \emph{stable-set polytope} associated to $G$ is the convex hull of the stable sets on $G$:
\[
    \textrm{Stab}(G) := \conv\left\{ s \in \R^V \mid s \textrm{ is a stable set on } G \right\}.
\]

Matchings and stable sets are closely related: every matching on $G$ corresponds to a stable set on the \emph{line graph} $L(G)$.
\begin{definition} \label{def: line graph}
    Let $G$ be a graph. The corresponding \emph{line graph} $L(G)$ has $V(L(G)) := E(G)$. Two vertices in $L(G)$ are linked by an edge if and only if the corresponding edges are adjacent in $G$.
\end{definition}

This correspondence reveals that matching polytopes are special stable-set polytopes, so we immediately obtain following proposition:
\begin{proposition} \label{prop: matchings are stable sets}
    For a graph $G$, the matching polytope of $G$ is isomorphic to the stable-set polytope of $L(G)$.
\end{proposition}

When $G$ is perfect, the stable-set polytope $\textrm{Stab}(G)$ has the IDP\cite{ohsugi2001compressed}.
Thus, when $L(G)$ is perfect, the matching polytope $P_M(G)$ has the IDP.
The line graph $L(G)$ is perfect if and only if every block of $G$ is either bipartite, $K_4$, or $K_{1, 1, n}$ \cite{MR0457293,maffray1992kernels}.
Thus, we also obtain the following proposition:

\begin{proposition}\label{prop:stab IDP}
    Suppose every block of $G$ is either bipartite, $K_4$, or $K_{1, 1, n}$.
    Then the matching polytope $P_M(G)$ has the IDP.
\end{proposition}

There exist, however, many matching polytopes with the IDP whose associated graphs are not of this form.
To prove that these polytopes possess the IDP we will need the minimal inequality description of the matching polytope.

\begin{definition} \label{def: essential vertex, factor critical}
    Given a vertex $u \in V(G)$, let $\deg_G(u)$ denote the degree of $v$ in $G$ and let $\iota(u)$ denote the set of edges incident to $u$ in $G$.
    We say that $u \in V$ is \emph{essential} if one of the following three conditions is satisfied:
    \begin{enumerate}
        \item $\deg_G(u)=1$ and the degree of its neighbor is also $1$,
        \item $\mathrm{deg}_G(u)=2$ and its neighbors are not adjacent, or
        \item $\deg_G(u) \ge 3$.
    \end{enumerate}
    A graph $G$ is \emph{factor-critical} if for every $u \in V$, $G \setminus u$ has a perfect matching.
\end{definition}

\begin{proposition}[{\cite{PulleyblankEdmonds}}]\label{prop:inequality}
The matching polytope $P_M(G)$ can be described as
\begin{align*}\label{inequality_description}
    P_M(G) 
    =
    \left\{ 
    x \in \mathbb{R}^{E} \;\;\; \Bigg{|} \;
    \begin{array}{rlc}
        x(e) &\geq 0, & \text{ for every } e \in E \vspace{0.2cm}\\
        \displaystyle \sum_{e \in \iota(u)} x(e) &\leq 1, &  \text{ for every  essential vertex } u \vspace{-0.3cm}\\ \\ 
        \displaystyle \sum_{e \in U} x(e) &\leq \frac{|U| - 1}{2}, &  \text{ for every $2$-connected induced} \vspace{-0.4cm}\\
        \; & \; &  \text{ factor-critical subgraph $U$}
    \end{array}
    \right\}.
\end{align*}
\end{proposition}


\section{Gorenstein matching polytopes}\label{sec:gorenstein}
In this section, we characterize when a matching polytope is Gorenstein and show that Gorenstein matching polytopes have the IDP.

\subsection{Gorenstein characterization} \label{Gorenstein characterization}
First, we give the following lemma:

\begin{lemma}\label{lem_gor}
Let $G$ be a connected graph and suppose that there exists a positive integer $\delta$ such that the following conditions hold:
\begin{itemize}
\item[(C1)]\label{C1} All essential vertices of $G$ have the same degree $\delta$;
\item[(C2)]\label{C2} For any $2$-connected induced factor-critical subgraph $H$, one has \[|E(H)|=(\delta+1)\frac{|H|-1}{2}-1.\]
\end{itemize}
Then,  
\begin{itemize}
\item[(i)] for $\delta = 3$, if $G$ has an odd cycle $C$, then the length of $C$ is at most $5$.
Moreover, if $G$ contains $C_5$, then $G[V(C_5)]=C'_5$.
\item[(ii)] for $\delta \ge 4$, then $G$ is bipartite.
\end{itemize}
\end{lemma}

\begin{proof}
For an odd cycle $C$ of $G$, let $H_C:=G[V(C)]$.
Note that $H_C$ is 2-connected and factor-critical.
\\ 

\noindent (i) Suppose $\delta = 3$, and let $C$ be an odd cycle of $G$.
Then by (C1): \[
\deg_{H_C}(u)\le \deg_G(u)\le 3 \text{ for any } u\in V(G).\]
Moreover, by (C2) and the Handshake Lemma:
\[2\left(4\cdot \frac{|H_C|-1}{2}-1\right)=\; 2|E(H_C)|=\;\sum_{u \in H_C}\deg_{H_C}(u)\le 3|H_C|.\]
Therefore, since $|C| = |H_C|$, we have $4|C|-6\le 3|C|$, so the length of $C$ is at most 5.

Furthermore, if $G$ has an odd cycle $C_5$, then:
\[|E(H_{C_5})|=4\cdot \frac{|H_{C_5}|-1}{2}-1=7,\]
that is, $C_5$ has just two chords and we have $H_{C_5}=C'_5$ since $\deg_G(u)\le 3$ for any $u\in V(G)$.
\\

\noindent (ii) Now let $\delta \geq 4$.
Suppose for contradiction that $G$ has an odd cycle.
Let $C_{2p+1}$ be a shortest odd cycle of $G$.
From (C2), we have:
\begin{align*}
|E(H_C)|&=(\delta+1)\left(\frac{|H_C|-1}{2}\right)-1 \\
&=(\delta+1)p-1>2p+1=|E(C_{2p+1})|.
\end{align*}
Thus, $C_{2p+1}$ has a chord and $G$ contains a shorter odd cycle than $C_{2p+1}$, a contradiction.
Therefore, $G$ has no odd cycles, equivalently, $G$ is bipartite.

\end{proof}

\begin{lemma}\label{lemma:gor}
Let $G$ be a connected graph.
Then the following are equivalent:
\begin{itemize}
\item[(i)] $P_M(G)$ is Gorenstein.

\item[(ii)] There exists a positive integer $\delta$ such that conditions (C1) and (C2) from Lemma \ref{C1} hold.
\end{itemize}
\end{lemma}

\begin{proof}
Assume that $P_M(G)$ is Gorenstein.
By Theorem~\ref{refl}, there exist $k\in \Z_{>0}$ and $\alpha \in k P\cap \Z^{E}$ such that $\alpha$ has height $1$ over each facet of $P_M(G)$.
From Proposition~\ref{prop:inequality}, $k P_M(G)$ has the facet defined by the inequality $x(e)\ge 0$ for each $e\in E$.
Thus, $\alpha$ satisfies $\alpha(e)=1$ for all $e\in E$.
Moreover, for each essential vertex $u$ and each $2$-connected induced factor-critical subgraph $H$ of $G$, $k P_M(G)$ has the facets corresponding to the inequalities:
\[\sum_{e \in \iota(u)} x(e) \leq k \text{ and } 
\sum_{e \in E(H)} x(e) \leq k \left( \frac{|H| - 1}{2}\right).\]
Thus, we can see that:
\begin{align*}
\deg_G(u) = \sum_{e\in \iota(u)}\alpha(e)=k-1, \; \text{ and } \\
|E(H)|=\sum_{e\in E(H)}\alpha(e)= k \left( \frac{|H| - 1}{2} \right) -1.
\end{align*}
Therefore, all essential vertices have the same degree $\delta=k-1$ and the equality: \[|E(H)|=(\delta+1)\left( \frac{|H|-1}{2} \right) -1\] holds for each $2$-connected induced factor-critical subgraph $H$.

The converse also follows from Theorem~\ref{refl}.
\end{proof}

The following result characterizes Gorenstein matching polytopes:

\begin{theorem}\label{thm:gor}
Let $G$ be a connected graph.
$P_M(G)$ is Gorenstein if and only if $G$ is one of the following graphs:
\begin{itemize}
\item[(a)]\label{class:a} A bipartite graph whose essential vertices have the same degree;
\item[(b)] \label{class:b}$C_5$;
\item[(c)]\label{class:c} A graph whose essential vertices have the same degree $3$ and whose blocks are bipartite, $K_3$, $K_4$, $K_{1,1,2}$ or $C'_5$.
\end{itemize}
\end{theorem}

\begin{proof}

From Lemma~\ref{lemma:gor}, $P_M(G)$ is Gorenstein if and only if there exists a positive integer $\delta$ such that conditions (C1) and (C2) hold.
We can easily find such $\delta$ for graphs of type (a), (b), and (c). 
Therefore, it is enough to show that if there exists a positive integer $\delta$ such that conditions (C1) and (C2) hold, then $G$ is of type (a), (b), or (c).

First, if $\delta \le 2$, then $G$ is a path or a cycle since $\deg_G(u) \le 2$ for any vertex $u$.
We see that $C_{2p+1}$ with $p\ge 3$ does not satisfy (C2), thus $G$ must be of type (a), (b), or (c).

Next, suppose $\delta=3$.
If $G$ is bipartite, then $G$ is of type (a), so we may assume that $G$ has an odd cycle.
Note that the lengths of odd cycles of $G$ are $3$ or $5$ by Lemma~\ref{lem_gor} (i).
We will show that $G$ is of type (c).
Take a block $B$ of $G$ which has an odd cycle.
If all odd cycles of $B$ have the same length $3$, then $B=K_4$ or $K_{1,1,n}$ with $n\le 2$ by (C1).

Suppose that $B$ contains $C_5$.
Then we can see that $H_{C_5}=C'_5$ by Lemma~\ref{lem_gor} (i) and that $B$ has an open ear decomposition $C\cup \mathtt{P}_1\cup \cdots\cup \mathtt{P}_r$ such that $C\cup \mathtt{P}_1\cup \mathtt{P}_2=C'_5$ for some $r\in \Z_{\ge 0}$ from Theorem~\ref{thm_factor}.
Each path $\mathtt{P}_i$ has length $2$, since $B$ cannot contain a cycle of length greater than $5$.
Thus, if $r\ge 3$, then there exists a vertex $u$ of $C_5$ with $\deg_B(u)\ge 4$, a contradiction to (C1).
Therefore, we must have $B=C'_5$, and hence $G$ is of type (c).

Finally, assume that $\delta \ge 4$.
Then $G$ is bipartite by Lemma~\ref{lem_gor} (ii), and hence $G$ is of type (a).

\end{proof}


\subsection{Gorenstein Matching Polytopes have the IDP}

We now determine when Gorenstein matching polytopes also have the IDP.
We shall use the following result from the literature:

\begin{lemma}[{\cite[Theorem~8.1]{engstrom2013ideals}}] \label{lemma: almost bipartite}
    The stable-set polytope of an almost bipartite graph has the IDP.
\end{lemma}

By Proposition \ref{prop:stab IDP}, the matching polytopes of bipartite graphs, $K_3$, $K_4$, $K_{1, 1, 2}$, and paths all have the IDP.
Furthermore, by Proposition \ref{prop: matchings are stable sets} $P_M(G) \cong \textrm{Stab}(L(G))$.
By Lemma \ref{lemma: almost bipartite}, since the line graph of a cycle is a cycle, cycles have the IDP.
Thus:

\begin{proposition} \label{big prop}
Graphs of type (a) and (b), as detailed in Theorem \ref{thm:gor}, give rise to matching polytopes with the IDP.
Furthermore, bipartite graphs, $K_3$, $K_4$, $K_{1, 1, 2}$, paths, and cycles give rise to polytopes with the IDP.
\end{proposition}

It turns out that graphs of type (c) also give rise to matching polytopes with the IDP.
We will obtain this result by breaking the graphs of type (c) into manageable graphs that we already know have the IDP.

Let $H$ be a subgraph of $G$, $x \in \R^{E(G)}$, and $X \subseteq \R^{E(G)}$.
We denote the restriction of $x$ to $\R^{E(H)}$ by $x|_H$ and we set $X|_H := \left\{ x|_H \colon x \in X \right\}$.

\begin{lemma}\label{lem_block}
    Let $H$ be a subgraph of $G$.
    Then for every positive integer $t$, $\left(tP_M(G)\right)|_H = tP_M(H)$.
\end{lemma}

\begin{proof}

Let $m$ be a matching on $G$.
Then each vertex of $H$ is incident to at most one edge of $m$, so $m|_H$ is a matching on $H$.

Conversely, let $\tilde{m}$ be a matching on $H$. Define $m \in \R^{E(G)}$ by
\begin{equation*}
    m(e) :=
    \begin{cases}
        \tilde{m}(e) & \textrm{ if } e \in E(H) \\
        0 & \textrm{ if else}.
    \end{cases}
\end{equation*}
Then $m$ is a matching of $G$ such that $m|_H = \tilde{m}$.
It follows that $P_M(G)|_H = P_M(H)$.
Furthermore, restriction commutes with dilation, so $\left( t P_M(G) \right)|_H = t P_M(H)$, as desired.

\end{proof}

\begin{proposition}\label{prop_block}
Let $G$ be a connected graph and let $B_1,\ldots,B_r$ be the blocks of $G$.
Then $P_M(B_i)$ has the IDP for each $i=1,\ldots,r$ if and only if $P_M(G)$ has the IDP.
\end{proposition}

\begin{proof}
First suppose $P_M(G)$ has the IDP. We would like to show that $P_M(B_i)$ has the IDP for each $i=1,\ldots,r$.
By Lemma \ref{lem_block}:
\begin{equation*}
    F_i := \left\{x\in P_M(G) : \sum_{e\in E(G)\setminus E(B_i)}x(e)=0\right\} \cong P_M(G)|_{B_i} = P_M(B_i).
\end{equation*}
Notice that $F_i$ is a face of $P_M(G)$ and thus has the IDP.
In fact, for $\alpha\in tF_i\cap \Z^{E(G)}$, if there exist $\alpha_1, \ldots, \alpha_t \in P_M(G) \cap \Z^{E(G)}$ with $\alpha = \alpha_1 + \cdots + \alpha_t$, then $\alpha_1,\ldots,\alpha_t$ must lie on $F_i$.
Therefore, $P_M(B_i)$ has the IDP for each $i=1,\ldots,r$, as desired.

Now, suppose $P_M(B_i)$ has the IDP for each $i=1,\ldots,r$.
We prove that $P_M(G)$ has the IDP via induction on $r$.
This result is evident when $r = 1$, so suppose $r \geq 2$.

Without loss of generality, $B_r$ contains just one cut vertex $u$ of $G$.
Let $H:=B_1\cup \cdots \cup B_{r-1}$ and notice that
\begin{equation*}
    V(B_r)\cap V(H)= \{u\} \text{ and } E(B_r)\cap E(H)= \emptyset.
\end{equation*}
Take $x\in tP_M(G)\cap \Z^{E(G)}$.
Let $y := x|_{B_r}$ and $\gamma := x|_H$.
By Lemma \ref{lem_block}:
\begin{align*}
    y &\in tP_M(B_r)\cap\Z^{E\left(B_r\right)}, \\
    \gamma &\in tP_M(H)\cap\Z^{E(H)}.
\end{align*}
By the induction hypothesis, $P_M(B_r)$ and $P_M(H)$ both have the IDP, so there exist matchings $y_1,\ldots,y_t$ on $B_m$ and $\gamma_1,\ldots,\gamma_t$ on $H$ such that:
\begin{align*}
    y&=y_1+\cdots+y_t, \\
    \gamma&=\gamma_1\cdots+\gamma_t.
\end{align*}
Let $\tilde{y}_i$ and $\tilde{\gamma}_i$ be defined as in Lemma \ref{lem_block}.
Notice that $\tilde{y}_i + \tilde{\gamma}_i$ is a matching on $G$ if and only if at most one of $\tilde{y}_i$ and $\tilde{\gamma}_i$ saturates $u$.
Thus, after a change of indices, $\tilde{y}_i+\tilde{\gamma}_i$ is a matching on $G$ for each $i=1,\ldots,t$.
Otherwise we would have:
\begin{equation*}
    \sum_{e\in \iota(v)}x(e)=\sum_{e\in \iota(v)}\left(\sum_{i=1}^t(\tilde{y}_i+\tilde{\gamma}_i)(e)\right)>t,
\end{equation*}
a contradiction.
Therefore, we can write:
\begin{equation*}
    x=(\tilde{y}_1+\tilde{\gamma}_1)+\cdots+(\tilde{y}_t+\tilde{\gamma}_t),
\end{equation*}
so $P_M(G)$ has the IDP.
\end{proof}

Proposition \ref{big prop} tells us that bipartite blocks, $K_3$, $K_4$, and $K_{1, 1, 2}$ give rise to matching polytopes with the IDP.
Therefore, to show that graphs in class (c) give rise to matching polytopes with the IDP, it remains only to show that the chortling cycle $C'_5$ gives rise to a matching polytope with the IDP.

\begin{remark}
One can use a computational algebra system to confirm that $P_M(C'_5)$ has the IDP, as we did using {\tt MAGMA} \cite{MAGMA}.
We provide a proof that $P_M(C'_5)$ has the IDP for the sake of completion and in order to introduce techniques that can be used to show other matching polytopes (which {\tt MAGMA} may be unable to process) also have the IDP.
\end{remark}

We continue by introducing the following seviceable definitions:

\begin{definition} \label{def: odd structure}
    Let $G$ be a connected graph.
    An \emph{odd structure} in $G$ is a $2$-connected, induced, factor-critical subgraph $U \subseteq G$.
    
    A \emph{$t$-matching} on $G$ is an assignment $x(e)$ of integer weights to the edges $e$ of $G$ such that:
    \begin{align*}
        &\sum_{e \in \iota(u)} x(e) \leq t, \;\; \textrm{ for every } u \in G \textrm{ and} \\
        &\sum_{e \in E(U)} x(e) \leq t \left(\frac{|U|-1}{2}\right), \;\; \textrm{ for every } \textrm{odd structure } U \subseteq G.
    \end{align*}

    A $t$-matching $x$ on $G$ \emph{splits} if there exists a $1$-matching $m$ and a $(t-1)$-matching $y$ both on $G$ such that:
    \begin{equation*}
        x = m + y.
    \end{equation*}

    Let $x$ be a $t$-matching, and let $u \in W_n$. 
    The \emph{weight} of $u$ (under $x$) is:
    \begin{equation*}
        x(u) := \sum_{e \in \iota(u)} x(e).
    \end{equation*}
    Similarly, the weight of $U \subseteq W_n$ is:
    \begin{equation*}
        x(U) := \sum_{e \in E(U)} x(e).
    \end{equation*}
    The \emph{index} $I(U)$ of $U$ is given by:
    \begin{equation*}
        x(U):=
        \begin{cases}
            I(U) + (t-1) \times \frac{|U|-1}{2}, & \text{ when } |U| \text{ is odd}, \\
            I(U) + (t-1) \times \frac{|U|}{2}, & \text{ when } |U| \text{ is even}.
        \end{cases}
    \end{equation*}
    Additionally:
    \begin{itemize}
        \item We say $u \in W_n$ is \emph{tight} if $x(u) = t$.
        Analogously, we say $U \subseteq W_n$ is tight if $I(U) > 0$.

        \item  When $|U|$ is odd, if $I(U) = \frac{|U|-1}{2}$ we say that $U$ is \emph{full index}; otherwise we say that $U$ is \emph{partial index}.

        \item If $u \in W_n$ or $U \subseteq W_n$ is not tight, it is \emph{loose}.
    \end{itemize}

    A $t$-matching $x$ is \emph{degenerate} if $x(e) > 0$ for some $e \in E(G)$. Otherwise, $x$ is \emph{nondegenerate}.
    
    Let $x$ be a $t$-matching on $G$.
    Then the subgraph $G[x]$ of $G$ \emph{induced} by $x$ is the largest subgraph $H$ of $G$ such that $x|_{E(H)}$ is nondegenerate.
    In particular:
    \begin{align*}
        V(G[x]) &= V(G), \\
        E(G[x]) &= \left\{ e \in E(G) \colon x(e) > 0 \right\}.
    \end{align*}
\end{definition}

Note that $t$-matchings on $G$ correspond to integer points in $t P_M(G)$, and vice versa.
Furthermore, $P_M(G)$ has no interior lattice points, so $1$-matchings are literal matchings.
We will henceforth use the two terms interchangeably.

\begin{lemma} \label{lem: Inductive Method}
    Let $G$ be a graph, and suppose that for every subgraph $H \subseteq G$, every nondegenerate matching on $H$ splits. Then $P_M(G)$ has the IDP.
\end{lemma}

\begin{proof}
    We want to prove every $t$-matching on $G$ can be written as the sum of $t$ matchings.
    By the principles of induction, it suffices to show every $t$-matching on $G$ splits.
    A $t$-matching $x$ on $G$ splits if and only if $x|_{G[x]}$ splits.
    Since $x|_{G[x]}$ is nondegenerate, it splits by assumption, and the result follows as desired.
\end{proof}

\begin{corollary} \label{cor: Inductive Method}
    Let $G$ be a graph, and suppose every proper subgraph of $G$ gives rise to a matching polytope with the IDP.
    If every nondegenerate $t$-matching on $G$ splits, then $P_M(G)$ has the IDP.
\end{corollary}

We will use this corollary to show $P_M(C'_5)$ has the IDP.
Recall, by Proposition \ref{big prop}, if the largest cycle in a graph $G$ is of length less than $5$, then $P_M(G)$ has the IDP.
It follows that almost every proper subgraph of $C'_5$ gives rise to a matching polytope with the IDP.
We treat the remaining case with the following lemma:

\begin{lemma} \label{Sublemma}
    Let $C^*_5$ denote the cycle of length $5$ with one chord. 
    Then $P_M(C^*_5)$ has the IDP.
\end{lemma}

\begin{proof}

    By Proposition \ref{big prop}, every proper subgraph of $C^*_5$ gives rise to a matching polytope with the IDP.
    Thus, by Corollary \ref{cor: Inductive Method}, it suffices to show every nondegenerate $t$-matching $x$ on $C^*_5$ splits.

    Let $x$ be a nondegenerate $t$-matching on $C^*_5$, and adopt the notation depicted in Figure \ref{fig:c5*}.
    We want to find a matching $m$ on $C^*_5$ such that $x-m$ is a $(t-1)$-matching on $C^*_5$.

    Suppose first that both shoulders of $C^*_5$ are tight.
    If both feet are tight, neither the top vertex nor the top triangle can be tight - otherwise the sum over all edge weights is at least $t$ (from the top) plus $2t-x(c)$ (from the two feet), which is strictly greater than $2t$.
    (We cannot have $x(c) = t$, because this would imply $a = 0$.)
    This contradicts the odd cycle condition enforced collectively by all the vertices in $C^*_5$.
    Thus, in this case we may choose $\{a, a'\}$ as our matching.

    Now, suppose without loss of generality that the left shoulder of $C^*_5$ is not tight.
    Then we may choose $\{b, c\}$ as our matching.
    
    In all cases, we may write (via induction) $x$ as the sum of $t$ matchings on $C^*_5$. It follows that $C^*_5$ has the IDP, as desired.
\end{proof}

\begin{lemma} \label{lem: Chortling Cycle}
    The matching polytope $P_M(C'_5)$ has the IDP.
\end{lemma}

\begin{proof}
    By Lemma \ref{Sublemma} and Proposition \ref{big prop}, every proper subgraph of $C'_5$ gives rise to a matching polytope with the IDP.
    Thus, by Corollary \ref{cor: Inductive Method}, it suffices to show every nondegenerate $t$-matching $x$ on $C^*_5$ splits.

    Let $x$ be a nondegenerate $t$-matching on $C'_5$ and adopt the notation depicted in Figure \ref{fig:c5'}. 
    We want a matching $m$ on $C'_5$ such that $x-m$ is a $(t-1)$-matching on $C'_5$.

    Suppose first that the top vertex is tight.
    Then at most one foot and one triangle can be tight, otherwise the sum over all edge weights is at least $t$ (from the top) plus $2t-x(c)$ (from the bottom), which as in the proof of Lemma \ref{Sublemma} leads to contradiction.
    Without loss of generality, the right foot is not tight.
    If the right triangle is loose, we may choose $\{a, a'\}$ as our matching.
    If the left is loose, we may choose $\{b, b'\}$ instead.

    Now suppose the top vertex is not tight. 
    Then we may choose $\{b, c\}$ as our matching.
    As in the previous lemma, our result follows as desired.
\end{proof}

\begin{figure}[H]
    \centering
    \begin{minipage}{0.47\textwidth}
        \centering
        \begin{tikzpicture}
            \foreach \i/\color in {1/blue, 2/green, 3/green, 4/blue, 5/red} {
                \node[draw, circle, minimum size=8pt, inner sep=0pt, fill=\color] (P\i) at ({90+72*\i}:1.5) {};
            }
            \foreach \i in {1,...,5} {
                \pgfmathtruncatemacro{\nexti}{mod(\i,5)+1}
                \draw (P\i) -- (P\nexti);
            }
            \draw (P1) -- (P4);

            \coordinate (E1) at ($ (P5)!0.5!(P4) $);
            \coordinate (E2) at ($ (P4)!0.5!(P3) $);
            \coordinate (E3) at ($ (P3)!0.5!(P2) $);
            \coordinate (E4) at ($ (P2)!0.5!(P1) $);

            \node at (E1) [above right] {$b$};
            \node at (E2) [right] {$a'$};
            \node at (E3) [below] {$c$};
            \node at (E4) [left] {$a$};

            \node[red] at (3, 1.3) {$\cdots$ Top};
            \node[blue] at (3, 0.5) {$\cdots$ Shoulders};
            \node[green] at (3, -1.2) {$\cdots$ Feet};
        \end{tikzpicture}
        \caption{The graph $C_5^*$}
        \label{fig:c5*}
    \end{minipage}
    \begin{minipage}{0.47\textwidth}
        \centering
        \begin{tikzpicture}
            \foreach \i/\color in {1/blue, 2/green, 3/green, 4/blue, 5/red} {
                \node[draw, circle, minimum size=8pt, inner sep=0pt, fill=\color] (P\i) at ({90+72*\i}:1.5) {};
            }
            \foreach \i in {1,...,5} {
                \pgfmathtruncatemacro{\nexti}{mod(\i,5)+1}
                \draw (P\i) -- (P\nexti);
            }
            \draw (P1) -- (P3);
            \draw (P2) -- (P4);

            \coordinate (E1) at ($ (P5)!0.5!(P4) $);
            \coordinate (E2) at ($ (P4)!0.5!(P2) $);
            \coordinate (E3) at ($ (P3)!0.5!(P1) $);
            \coordinate (E4) at ($ (P2)!0.5!(P1) $);
            \coordinate (E5) at ($ (P5)!0.5!(P1) $);

            \node at (E1) [above right] {$a'$};
            \node at (E2) [above] {$b$};
            \node at (E3) [above] {$c$};
            \node at (E4) [left] {$a$};
            \node at (E5) [above left] {$b'$};

            \node[red] at (3, 1.3) {$\cdots$ Top};
            \node[blue] at (3, 0.5) {$\cdots$ Shoulders};
            \node[green] at (3, -1.2) {$\cdots$ Feet};
        \end{tikzpicture}
        \caption{The graph $C'_5$}
        \label{fig:c5'}
    \end{minipage}
\end{figure}

\begin{corollary}\label{cor:gorenstein_idp}
    Let $G$ be a connected graph.
    If $P_M(G)$ is Gorenstein, then it has the IDP.
\end{corollary}

\begin{proof}
    Theorem~\ref{thm:gor} tells us that $G$ is of form (a), (b) or (c).
    Proposition \ref{big prop} tells us that if $G$ is of form (a) or (b), then $P_M(G)$ gives rise to a matching polytope with the IDP.
    Thus, we may assume $G$ is of form (c).

    To show that $G$ has the IDP, it suffices to show that every block of $G$ gives rise to a matching polytope with the IDP (Proposition \ref{prop_block}).
    Since $G$ is of form (c), each block of $G$ is either bipartite, $K_3$, $K_4$, $K_{1, 1, 2}$, or $C'_5$.
The matching polytopes of bipartite graphs, $K_3$, $K_4$ and $K_{1,1,2}$ have the IDP by Proposition~\ref{big prop}.
    The matching polytope of $C'_5$ has the IDP by Lemma~\ref{lem: Chortling Cycle}.
    The result follows, as desired.
\end{proof}


\section{Wheel graphs are not Gorenstein, but have the IDP }
\label{sec:wheel_graphs_IDP}

In Section \ref{sec:gorenstein}, we characterized Gorenstein matching polytopes and proved they have the IDP.
We found that, to prove a matching polytope $P_M(G)$ has the IDP, it suffices to split every $t$-matching $x$ on $G$ into a matching $m$ and a $(t-1)$-matching $y$.
It is generally challenging to find good candidates for $m$ when some $e \in E(G)$ has $x(e) = 0$.
However, in this scenario, it suffices to instead show $P_M\left( G \setminus e \right)$ has the IDP - an easier task.

Unfortunately, this inductive method (Corollary \ref{cor: Inductive Method}) does not suffice to characterize the matching polytopes with IDP.
For instance, more advanced techniques are required to show that the matching polytope of the wheel graph has the IDP.
Notice first, that matching polytopes of wheel graphs are generally not Gorenstein:

\begin{proposition}
    Let $n \geq 5$. Then $P_M (W_n)$ is not Gorenstein.
\end{proposition}

\begin{proof}
    By Theorem \ref{thm:gor}, $P_M (W_n)$ is Gorenstein if and only if one of the following holds:
    \begin{enumerate}
        \item $W_n$ is a bipartite graph whose essential vertices have the same degree; \label{point 1}
        \item $W_n = C_5$; \label{point 2}
        \item All essential vertices of $W_n$ have the same degree, and each block of $W_n$ is either $K_3$, $K_4$, $K_{1, 1, 2}$, or $C'_5$. \label{point 3}
    \end{enumerate}
    Since every vertex in $W_n$ has degree at least $3$, every vertex is essential as in Definition \ref{def: essential vertex, factor critical}. 
    For $n \geq 5$, the vertices on the outer cycle of $W_n$ all have degree $3$, but the central vertex has degree $n-1 \neq 3$.
    Thus $P_M (W_n)$ does not satisfy (\ref{point 1}) or (\ref{point 3}).
    Clearly, $W_n \neq C_5$, so $P_M (W_n)$ does not satisfy (\ref{point 2}) either.
    It follows that $P_M (W_n)$ is not Gorenstein.
\end{proof}

Nevertheless, by the end of this section we prove that the matching polytope of the wheel graph has the IDP. 

By Lemma \ref{lem: Inductive Method}, it suffices to show every nondegenerate $t$-matching $x$ on each subgraph $H$ of $W_n$ splits.
To this end, let $H$ be a subgraph of $W_n$.
We characterize the $t$-matchings on $H$.
By Definition \ref{def: odd structure}, it suffices to characterize the odd structures in $H$.
With this in mind, we introduce the following definitions:

\begin{definition} \label{def: slices}
    Let $v$ and $\Delta$ denote the central vertex the outer cycle of $W_n$, respectively.
    A \emph{slice} $S$ of $H$ is a $2$-connected induced subgraph of $H$ containing $v$.
    A slice $S$ is \emph{even} if $|S|$ is even; otherwise $S$ is \emph{odd}.

    Let $O(S) := S \cap \Delta$ denote the \emph{boundary} of $S$.
    Note that $O(S)$ is either a path or a cycle.
    Two slices $S$ and $T$ are \emph{almost disjoint} if:
    \begin{equation*}
        O(S) \cap O(T) = \emptyset \text{(see Figure \ref{fig:fig}(A) for an example)}.
    \end{equation*}
    These slices are \emph{somewhat disjoint} if:
    \begin{equation*}
        E[O(S) \cap O(T)] = \emptyset \text{ (see Figure \ref{fig:fig}(B) for an example) }.
    \end{equation*}
    Otherwise, $S$ and $T$ \emph{interlock} (see Figure \ref{fig:fig}(D) for an example).
\end{definition}

\begin{figure}
\begin{subfigure}{.45\textwidth}
  \centering
  \begin{tikzpicture}[scale=0.8,every node/.style={scale=0.8}]
    \graph  [empty nodes, nodes={circle,fill=black}, edges={black, thick}, clockwise, radius=8em,
    n=9, p=0.3] 
        { subgraph C_n [n=8,m=2,clockwise,radius=2cm,name=A]-- mid};
        \node at  ($(0,3.1)$) (C){};
        \foreach \i [count=\xi from 2]  in {1,...,8}{\node at (A \i){ };}
        \draw[ultra thick, red]  (A 4) -- (A 5) -- (C) -- (A 4) -- (A 3) -- (C) -- (A 2) -- (A 3);
        \draw[ultra thick, blue] (C) -- (A 8) -- (A 7) -- (A 6) -- (C) -- (A 7);
        \node[circle,draw=black,fill=black!100] () at (0,3.1) {};
\end{tikzpicture}
  \caption{Two almost disjoint slices of $W_9$: an even slice (blue) and an odd slice (red).}
  \label{fig:sfig1}
\end{subfigure}
\hfill
\begin{subfigure}{.45\textwidth}
  \centering
 \begin{tikzpicture}[scale=0.8,every node/.style={scale=0.8}]
    \graph  [empty nodes, nodes={circle,fill=black}, edges={black, thick}, clockwise, radius=8em,
    n=5, p=0.3] 
        { subgraph C_n [n=4,m=2,clockwise,radius=2cm,name=A]-- mid};
        \node at  ($(0,3.1)$) (C){};
        \foreach \i [count=\xi from 2]  in {1,...,4}{\node at (A \i){ };}
        \draw[ultra thick, red]  (A 2) -- (A 3) -- (C) -- (A 2);
        \node at  ($(0,3.15)$) (C'){};
        \node at  ($(2,3.15)$) (A 2'){};

        \draw[ultra thick, blue] (C') -- (A 2');
        \draw[ultra thick, blue] (C') -- (A 1) -- (A 2);

\node[circle,draw=black,fill=black!100] () at (2,3.1) {};

\node[circle,draw=black,fill=black!100] () at (0,3.1) {};
\end{tikzpicture}
  \caption{Two somewhat (but not almost) disjoint slices of $W_5$.}
  \label{fig:sfig2}
\end{subfigure}
\begin{subfigure}{.45\textwidth}
  \centering
 \begin{tikzpicture}[scale=0.8,every node/.style={scale=0.8}]
    \graph  [empty nodes, nodes={circle,fill=black}, edges={black, thick}, clockwise, radius=8em,
    n=7, p=0.3] 
        { subgraph C_n [n=6,m=2,clockwise,radius=2cm,name=A]-- mid};
        \node at  ($(0,3.1)$) (C){};
        \foreach \i [count=\xi from 1]  in {1,...,6}{\node at (A \i){ };}
        \draw[ultra thick, red]  (A 1) -- (A 2) -- (C) -- (A 4) -- (A 3) -- (C) -- (A 5) -- (A 6) -- (C) -- (A 1);

\node[circle,draw=black,fill=black!100] () at (0,3.1) {};
\end{tikzpicture}
  \caption{An odd bundle of slices of $W_7$.}
  \label{fig:sfig3}
\end{subfigure}
\hfill
\begin{subfigure}{.45\textwidth}
  \centering
 \begin{tikzpicture}[scale=0.8,every node/.style={scale=0.8}]
    \graph  [empty nodes, nodes={circle,fill=black}, edges={black, thick}, clockwise, radius=8em,
    n=7, p=0.3] 
        { subgraph C_n [n=6,m=2,clockwise,radius=2cm,name=A]-- mid};
        \node at  ($(0,3.1)$) (C){};
        \foreach \i [count=\xi from 1]  in {1,...,6}{\node at (A \i){ };}
         \draw[ultra thick, blue]  (A 1) -- (A 2) -- (A 3) -- (A 4) -- (A 1);
        \draw[ultra thick, red]  (A 3) -- (A 4) -- (A 5) -- (A 6) -- (A 3);
        
  \node at  ($(0,1.15)$) (A 4'){};
        \node at  ($(1.69,2.15)$) (A 3'){};
        
\draw[ultra thick, blue]  (A 3') -- (A 4');

\node[circle,draw=black,fill=black!100] () at (0,3.1) {};

\end{tikzpicture}
  \caption{Two interlocking slices. }
\end{subfigure}
\caption{Figures depicting: almost disjoint slices (A), somewhat disjoint slices (B), an odd bundle of slices (C), and interlocking slices (D).}
\label{fig:fig}
\end{figure}
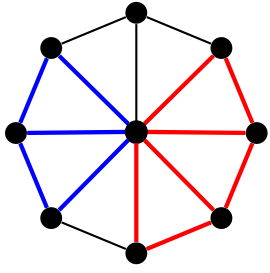
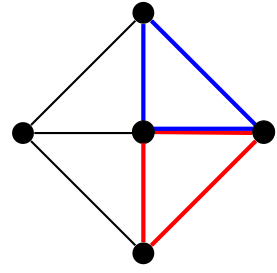
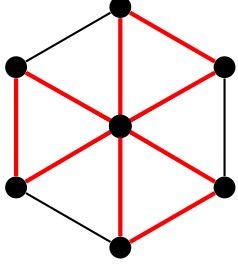
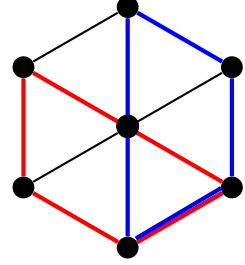

\begin{lemma}
    A subgraph $U \subseteq H \subseteq W_n$ is an odd structure of $H$ if and only if either 
    \begin{enumerate}
        \item \label{item1}  $U$ is an odd slice of $H$ or
        \item \label{item2}  $U = \Delta$ and $n$ is even.
    \end{enumerate}
\end{lemma}

\begin{proof}
    Suppose $U \subseteq H$ is an odd structure in $H$.
    Then $U$ is induced and $2$-connected.
    Thus, if $v \in U$, then $U$ is a slice of $H$; otherwise $U = \Delta$.
    A slice of $H$ is factor critical if and only if it is odd, and $\Delta$ is factor critical if and only if $n$ is even.
    Thus, either \emph{(1)} $U$ is an odd slice or \emph{(2)} $U = \Delta$ and $n$ is even, as desired.
    The reverse direction follows by similar principles.
\end{proof}

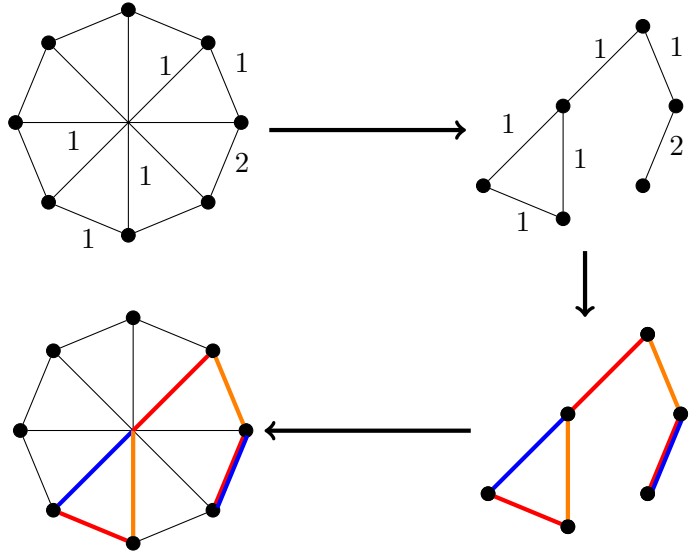
\begin{figure}
    \centering
    \begin{tikzpicture}

\node (a) at (-6, 0)
{
    \begin{tikzpicture}[scale=1]

\foreach \i in {1,...,8}
    \node[fill,circle,draw, scale=0.5] (v\i) at ({360/8 * (\i - 1)}:1.5) {};

\foreach \i in {1,...,7}
    \draw (v\i) -- (v\the\numexpr\i+1\relax);
\draw (v8) -- (v1);

\foreach \i in {1,...,8}
    \draw (v\i) -- (0,0);

\path (0,0) -- node[midway, above ] {1} (v2);
\path (v1) -- node[midway, above right] {1} (v2);

\path (0,0) -- node[midway, right ] {1} (v7);
\path (0,0) -- node[midway, above left] {1} (v6);
\path (v6) -- node[midway, below ] {1} (v7);
\path (v8) -- node[midway, right] {2} (v1);

\end{tikzpicture}
};

\node (b) at (0,0)
{
    \begin{tikzpicture}[scale=1]

\node[fill,circle,draw, scale=0.5] (v8) at ({360/8 * (8 - 1)}:1.5) {};

\node[fill,circle,draw, scale=0.5] (v1) at ({360/8 * (1 - 1)}:1.5) {};

\node[fill,circle,draw, scale=0.5] (v2) at ({360/8 * (2 - 1)}:1.5) {};

\node[fill,circle,draw, scale=0.5] (v6) at ({360/8 * (6 - 1)}:1.5) {};

\node[fill,circle,draw, scale=0.5] (v7) at ({360/8 * (7 - 1)}:1.5) {};

\node[fill,circle,draw, scale=0.5] at (0,0) {};

\draw (v8) -- (v1) -- (v2) -- (0,0) -- (v7) -- (v6) -- (0,0);

\path (0,0) -- node[midway, above ] {1} (v2);
\path (v1) -- node[midway, above right] {1} (v2);

\path (0,0) -- node[midway, right ] {1} (v7);
\path (0,0) -- node[midway, above left] {1} (v6);
\path (v6) -- node[midway, below ] {1} (v7);
\path (v8) -- node[midway, right] {2} (v1);
\end{tikzpicture}
};

\node (c) at (0,-4)
{
    \begin{tikzpicture}[scale=1]

\node[fill,circle,draw, scale=0.5] (v8) at ({360/8 * (8 - 1)}:1.5) {};

\node[fill,circle,draw, scale=0.5] (v1) at ({360/8 * (1 - 1)}:1.5) {};

\node[fill,circle,draw, scale=0.5] (v2) at ({360/8 * (2 - 1)}:1.5) {};

\node[fill,circle,draw, scale=0.5] (v6) at ({360/8 * (6 - 1)}:1.5) {};

\node[fill,circle,draw, scale=0.5] (v7) at ({360/8 * (7 - 1)}:1.5) {};

\node[fill,circle,draw, scale=0.5] at (0,0) {};

\node(v1')[scale=0.5] at ({360/8 * (1 - 1)}:1.5+0.045) {};

\node(v8')[scale=0.5] at ({360/8 * (8 - 1)}:1.5+0.045) {};

\draw (v8) -- (v1) -- (v2) -- (0,0) -- (v7) -- (v6) -- (0,0);

\draw[ultra thick, red] (v6) -- (v7);
\draw[ultra thick, blue] (v6) -- (0,0);
\draw[ultra thick, orange] (v7) -- (0,0);
\draw[ultra thick, red] (v2) -- (0,0);
\draw[ultra thick, orange] (v2) -- (v1);
\draw[ultra thick, red] (v1) -- (v8);
\draw[ultra thick, blue] (v1') -- (v8');

\path (0,0) -- node[midway, above ] {} (v2);
\path (v1) -- node[midway, above right] {} (v2);

\path (0,0) -- node[midway, right ] {} (v7);
\path (0,0) -- node[midway, above left] {} (v6);
\path (v6) -- node[midway, below ] {} (v7);
\path (v8) -- node[midway, right] {} (v1);

\node[fill,circle,draw, scale=0.5] (v8) at ({360/8 * (8 - 1)}:1.5) {};

\node[fill,circle,draw, scale=0.5] (v1) at ({360/8 * (1 - 1)}:1.5) {};

\node[fill,circle,draw, scale=0.5] (v2) at ({360/8 * (2 - 1)}:1.5) {};

\node[fill,circle,draw, scale=0.5] (v6) at ({360/8 * (6 - 1)}:1.5) {};

\node[fill,circle,draw, scale=0.5] (v7) at ({360/8 * (7 - 1)}:1.5) {};

\node[fill,circle,draw, scale=0.5] at (0,0) {};
\end{tikzpicture}
};

\node (d) at (-6,-4)
{
 \begin{tikzpicture}[scale=1]

\foreach \i in {1,...,8}
    \node[fill,circle,draw, scale=0.5] (v\i) at ({360/8 * (\i - 1)}:1.5) {};
    
\node(v1')[scale=0.5] at ({360/8 * (1 - 1)}:1.5+0.045) {};

\node(v8')[scale=0.5] at ({360/8 * (8 - 1)}:1.5+0.045) {};

\foreach \i in {1,...,7}
    \draw (v\i) -- (v\the\numexpr\i+1\relax);
\draw (v8) -- (v1);

\foreach \i in {1,...,8}
    \draw (v\i) -- (0,0);

\path (0,0) -- node[midway, above ] {} (v2);
\path (v1) -- node[midway, above right] {} (v2);

\path (0,0) -- node[midway, right ] {} (v7);
\path (0,0) -- node[midway, above left] {} (v6);
\path (v6) -- node[midway, below ] {} (v7);
\path (v8) -- node[midway, right] {} (v1);

\draw[ultra thick, red] (v6) -- (v7);
\draw[ultra thick, blue] (v6) -- (0,0);
\draw[ultra thick, orange] (v7) -- (0,0);
\draw[ultra thick, red] (v2) -- (0,0);
\draw[ultra thick, orange] (v2) -- (v1);
\draw[ultra thick, red] (v1) -- (v8);
\draw[ultra thick, blue] (v1') -- (v8');

\end{tikzpicture}
};
\draw[ultra thick] [->] (a)--(b);
\draw[ultra thick] [->] (b)--(c);
\draw[ultra thick] [->] (c)--(d);

\end{tikzpicture}

\caption{The top-left figure depicts a 3-matching on $W_9$, which becomes a similar 3-matching on a graph $G$ (top-right image), from which we can decompose into 3-matchings on $G$ (the yellow, red, and blue in the bottom-right image). And this tells us how to decompose our matching on $W_9$ (bottom-left image). }
\label{fig:zero-weight trick}
\end{figure}
Now, let $x$ be a nondegenerate $t$-matching on $H$.
We want to find a matching $m$ on $H$ such that $x - m$ is a $(t-1)$-matching on $H$.
Our matching $m$ must satisfy an inequality description:
\begin{enumerate}
    \item $m(u) = 1$ for each tight vertex $u \in H$, and
    \item $m(U) \geq I(U)$ for each tight odd structure $U \subseteq H$.
\end{enumerate}

Motivated by this inequality description, we introduce the following definitions:

\begin{definition}
    Let $x$ be a nondegenerate $t$-matching on $H$, and let $m$ be a matching on $H$.
    Then $m$ \emph{assigns sufficient weight} to a tight vertex $u \in H$ if $m$ saturates $u$, to an even slice $S \subseteq H$ if $m$ saturates every tight vertex in $S$, and to an odd structure $U \subseteq H$ if $m(U) \geq I(U)$.
\end{definition}

However, our inequality description is far from minimal. For example:

\begin{lemma}\label{lem:interlocking}
    Let $S$ and $T$ be interlocking tight odd slices.
    Let $R$ be the slice of $H$ given by:
    \begin{equation*}
    R := H[(O(S) \cap O(T)) \cup \{ v \}].
    \end{equation*}
    Notice that $O(R)$ has a two-vertex boundary $\{ u_S, u_T \}$, where $u_S$ is in the interior of $O(S)$ and $u_T$ is in the interior of $O(T)$.
    Let $S'$ and $T'$ be the slices of $H$ given by:
    \begin{align*}
        S' &:= H[(O(S) \setminus O(T)) \cup \{ v, u_S\}], \\
        T' &:= H[(O(T) \setminus O(S)) \cup \{ v, u_T \}].
    \end{align*}
    If a matching $m$ assigns sufficient weight to $R$, $S$, and $T'$, then $m$ automatically assigns sufficient weight to $S$ and $T$.
\end{lemma}

\begin{figure}
    \centering
\begin{tikzpicture}[scale=0.8,every node/.style={scale=0.8}]
    \graph  [empty nodes, nodes={circle,fill=black}, edges={black, thick}, clockwise, radius=8em,
    n=7, p=0.3] 
        { subgraph C_n [n=6,m=2,clockwise,radius=2cm,name=A]-- mid};
                \node at  ($(0,3.1)$) (C){};

          \node at  ($(0,1.15)$) (A 4'){};
        \node at  ($(1.69,2.15)$) (A 3'){};

\fill[orange!70] (A 3.center) -- (A 4.center) -- (C.center) -- cycle;

\fill[purple!40] (A 1.center) -- (A 2.center) --  (A 3.center) -- (C.center) -- cycle;

\fill[green!40] (A 4.center) -- (A 5.center) --  (A 6.center) -- (C.center) -- cycle;

        \foreach \i [count=\xi from 1]  in {1,...,6}{\node at (A \i){ };}
         \draw[ultra thick, blue]  (A 1) -- (A 2) -- (A 3) -- (A 4) -- (A 1);
        \draw[ultra thick, red]  (A 3) -- (A 4) -- (A 5) -- (A 6) -- (A 3);

\draw[ultra thick, blue]  (A 3') -- (A 4');

\draw[ black!50]  (C) -- (A 2);

\draw[ black!50]  (C) -- (A 5);

\node[circle,draw=black,fill=black!100] () at (0,3.1) {};

\node[circle,draw=black,fill=black!100] () at (A 3) {};

\node[circle,draw=black,fill=black!100] () at (A 4) {};

\node[circle,draw=black,fill=black!100] () at (A 1) {};

\node[circle,draw=black,fill=black!100] () at (A 2) {};

\node[circle,draw=black,fill=black!100] () at (A 5) {};

\node[circle,draw=black,fill=black!100] () at (A 6) {};

\end{tikzpicture}
\caption{An example of how two interlocking odd slices, $S$ (in {\color{red}{red}}) and $T$ (in {\color{blue}{blue}}), give rise to $S'$ (in {\color{green!40}{green}}), $R$ (in {\color{orange!70}{orange}}), and $T'$ (in {\color{purple!40}{pink}}). }
    \label{fig:Interlocking Slices}
\end{figure}
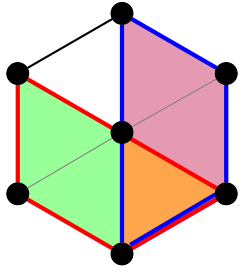

\begin{proof}
    Without loss of generality, $R$ is odd\footnote{The same techniques used to treat this case apply when $R$ is even.},
    so $S'$ and $T'$ are both even.    
    If $m$ assigns sufficient weight to $S'$, then $m(S') \geq I(S')$ by the Handshake Lemma and Pigeonhole Principle.
        
    Let $e$ denote the sole spoke in $S' \cap R$.
    This spoke exists in $H$ because it lies on the outer edge of $T$, and $T$ is $2$-connected.
    Since $x$ is nondegenerate we have:
    \begin{align*}
        I(S) &= x(S) - \frac{|S|-1}{2} \\
        &= x(S') + x(R) - x(e) - \left(\frac{|S'|}{2} + \frac{|R| - 1}{2} - \frac{1}{2}\right) \\
        &= I(S') + I(R) + \frac{1}{2} - x(e) \\
        &\leq I(S') + I(R).
    \end{align*}
    Thus, if $m$ assigns sufficient weight to $S'$ and $R$, it also assigns sufficient weight to $S$.
    
    Similarly, if $m$ assigns sufficient weight to $T'$ and $R$, it also assigns sufficient weight to $T$, as desired.
\end{proof}

Figure \ref{fig:Interlocking Slices} depicts how two interlocking odd slices give rise to $S'$, $R$, and $T'$ as instructed by Lemma \ref{lem:interlocking}.

It follows by the principles of induction that when we construct $m$ we need only consider some \emph{troublesome} collection of tight, somewhat disjoint odd slices (alongside the tight vertices).
Furthermore, it turns out that we need not consider tight vertices inside troublesome slices:

\begin{lemma}
    Let $S$ be a tight odd slice and $u$ a tight vertex in the interior of $O(S)$.
    Then $S \setminus \{ u \}$ is an almost disjoint union of two slices $R$ and $T$.
    If a matching $m$ assigns sufficient weight to $R$, $T$, and $u$, then $m$ automatically assigns sufficient weight to $S$.
\end{lemma}

\begin{proof}
    We have $x(S) = x(R) + x(T) + t$ and:
    \begin{equation*}
        \frac{|S|-1}{2} = \frac{|R|+|T|-1}{2}.
    \end{equation*}
    Since $R$, $T$, and the edges in $H$ incident to $u$ are all disjoint, it follows that $I(S) = I(R) + I(T) + 1$. Thus, if $m$ assigns sufficient weight to $R$, $T$, and $u$, then it assigns sufficient weight to $S$, as desired. 
\end{proof}

We are now ready to construct $m$.
To this end, let $S$ be a troublesome slice.
Note $O(S)$ contains an even number of vertices linked together by an odd number of edges.
We can choose to add at most $\frac{|S|-1}{2}$ of these edges to $m$.
If we choose this many edges, we sufficiently reduce the weight of $S$ (regardless of the index of $S$), but we also include in $m$ the edges at the outer ends of $O(S)$.
If the vertices in $H \setminus S$ near $S$ are tight, this can cause complications.

\begin{figure}
    \centering
    
\begin{tikzpicture}[scale=1.5]

\node (a) at (-3, 0)
{
    \begin{tikzpicture}[scale=1]

\foreach \i in {1,...,7}
    \node[fill,circle,draw, scale=0.5] (v\i) at ({-360/7 * (\i - 1) + 90}:1.5) {};

\foreach \i in {2,...,7}
    \draw (v1) -- (v\i);

\draw[ultra thick, red] (v2) -- (v3);
\draw (v3) -- (v4);
\draw[ultra thick, red] (v4) -- (v5);
\draw (v5) -- (v6);
\draw[ultra thick, red] (v6) -- (v7);

\end{tikzpicture}
};
\node (b) at (0,0)
{
  \begin{tikzpicture}[scale=1]

\foreach \i in {1,...,7}
     \node[fill,circle,draw, scale=0.5] (v\i) at ({-360/7 * (\i - 1) + 90}:1.5) {};

\foreach \i in {2,...,7}
    \draw (v1) -- (v\i);

\draw (v2) -- (v3);
\draw[ultra thick, blue] (v3) -- (v4);
\draw (v4) -- (v5);
\draw[ultra thick, blue] (v5) -- (v6);
\draw (v6) -- (v7);

\end{tikzpicture}
};
\node (c) at (3,0)
{
  \begin{tikzpicture}[scale=1]

\foreach \i in {1,...,7}
    \node[fill,circle,draw, scale=0.5] (v\i) at ({-360/7 * (\i - 1) + 90}:1.5) {};

\foreach \i in {2,...,7}
    \draw (v1) -- (v\i);

\draw[ultra thick, orange] (v2) -- (v3);
\draw (v3) -- (v4);
\draw[ultra thick, orange] (v4) -- (v5);
\draw (v5) -- (v6);
\draw (v6) -- (v7);

\end{tikzpicture}
};
\end{tikzpicture}
\caption{Left: Choosing $\frac{|S|-1}{2}$ edges of a troublesome slice $S$.\\
Center: Choosing $\frac{|S|-1}{2}-1$ edges so as to miss both the outer edges of $S$.\\
Right: Choosing $\frac{|S|-1}{2}-1$ edges so as to include one of the outer edges and miss another.}
    \label{fig:The Algorithm on W}
\end{figure}
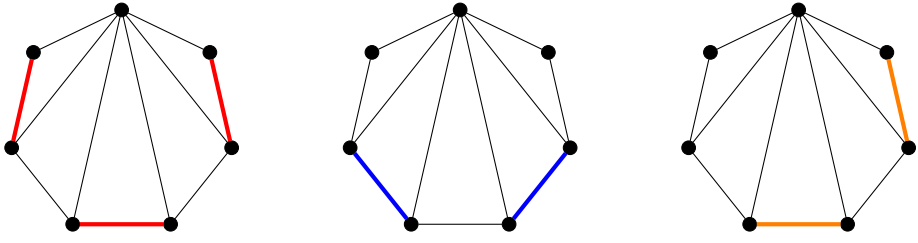

\begin{definition}
    Let $\Omega \subseteq \Delta \cap H$ be the interior of a path connecting two troublesome slices. 
    If $|\Omega|$ is odd and every vertex in $\Omega$ is tight, then $\Omega$ is a \emph{conductor}.
    Otherwise, $\Omega$ is an \emph{insulator}.
    
    We say two troublesome slices $S$ and $T$ are linked by a \emph{pseudoconductor} $\Omega$ if they are somewhat, but not almost disjoint.
    In this case, let $e$ be the edge shared by $S$ and $T$.
    We adopt the convention $e \in \Omega$ and $|\Omega| = -1$.
\end{definition}

If $S$ is not full index, (pseudo)conductors from $S$ to other troublesome slices do not pose any problems because we may choose $\frac{|S|-1}{2}-1$ alternating edges of $O(S)$ so as to either include exactly one, or none, of the outermost edges of $O(S)$. 
Furthermore, it turns out that if two full-index troublesome slices $S$ and $T$ are linked by a (pseudo)conductor $\Omega$, the troublesome cycles of $H$ must be so uncomplicated that the challenges posed by $\Omega$ are easily resolved\footnote{There is a quirk to all this: experimentation suggests that this scenario can never happen. 
However, it is easier to prove that this problem can be solved than it is to prove that it never appears.}.

\begin{lemma} \label{lem: two full index and a conductor}
    Suppose $H$ contains two full index troublesome slices $S$ and $T$ linked by a (pseudo)conductor $\Omega$.
    Then either 
    \begin{enumerate}
        \item $H$ contains no other troublesome slices or 
        \item $H$ contains exactly one other troublesome slice $W$. 
    \end{enumerate} 
    In case \emph{(1)}, $S$ and $T$ do not share a spoke outside of $\Omega$.
    In case \emph{(2)}, $W$ is somewhat but not almost disjoint from $S$ and $T$, and $W$ is only of partial index.
\end{lemma}

\begin{proof}
    Suppose $H$ contains two full index troublesome slices $S$ and $T$ linked by a conductor $\Omega$.
    Then $O(S)$, $\Omega$, and $O(T)$ are pairwise disjoint, and $|S| + |\Omega| + |T|$ is odd.
    Thus the slice:
    \begin{equation*}
        R := S \cup \Omega \cup T
    \end{equation*}
    is even.
    It follows that $O(R)$ contains an odd number of vertices.
    Consider the alternating pattern of vertices along $O(R)$ that does not contain either endpoint of $O(R)$.
    Let $D \subseteq R$ consist of $v$ plus every vertex in this alternating pattern.
    Every edge in $R$ is adjacent to at least one vertex in $D$, and at each vertex $u \in D$ we have $x(u) \leq t$, so:
    \begin{equation} \label{eq: SCT along O(R)}
        x(R) \leq t \; |D| = t \left(\frac{|S|+|\Omega|+|T|-1}{2}\right).
    \end{equation}

    At the same time, since $S$ and $T$ are of full index and $\Omega$ is a conductor, we must have:
    \begin{align*}
        x(R) &= t \left( \frac{|S|-1}{2} + \frac{|T|-1}{2} + \frac{|\Omega|+1}{2} \right) \\
        &= t \left( \frac{|S| + |\Omega| + |T| - 1}{2}\right).
    \end{align*}

    Thus, equality holds in Inequality (\ref{eq: SCT along O(R)}), so every vertex in $D$ is tight with respect to $x|_R$.
    In particular:
    \begin{equation*}
        t = x|_R (v) \leq x(v) \leq t,
    \end{equation*}
    so the only spokes in $H$ are those in $R$ (since $x$ is nondegenerate).
    Furthermore, since odd slices are $2$-connected, they always contain at least two spokes.
    Thus, $H$ cannot have more than three troublesome slices: $S$, $T$, and possibly $W$.
    If $W$ exists, it must share one spoke with $S$ and another with $T$.

    Similarly, suppose $H$ contains two full-index troublesome slices $S$ and $T$ that are linked by a pseudoconductor $\Omega$ containing the edge $e$.
    Let $R$ be the even slice given by:
    \begin{equation*}
        R := S \cup T.
    \end{equation*}
    As before, look at every other vertex along $O(R)$ and at the center of $W_n$, but this time notice that $e$ links one of the vertices considered along $O(R)$ to the center of $W_n$, so that:
    \begin{equation} \label{eq: back to back}
        x(R) \leq t \left(\frac{|R|}{2}\right) - x(e).
    \end{equation}
    At the same time, since $S$ and $T$ are full index, we must have:
    \begin{align*}
        x(R) &\geq t \left( \frac{|S|-1}{2} + \frac{|T|-1}{2} \right) - x(e) \\
        &= t \left( \frac{|R|}{2} - \right) - x(e).
    \end{align*}
    Thus, equality holds in Inequality (\ref{eq: back to back}).
    As before, it follows that $H$ cannot have more than three troublesome slices: $S$, $T$, and possibly $W$, and that if $W$ exists, it must share one spoke with $S$ and another with $T$.

    Suppose for contradiction that $W$ does not exist, and that $S$ and $T$ share a spoke outside of $\Omega$.
    Then $H$ is itself an odd slice, so:
    \begin{equation*}
        x(H) \leq t \left( \frac{|H|-1}{2} \right).
    \end{equation*}
    On the other hand:
    \begin{align*}
        x(H) &> t \left( \frac{|S|-1}{2} + \frac{|T|-1}{2} + \frac{|\Omega|+1}{2} - 1 \right) \\
        &= t \left( \frac{|H|-1}{2} \right).
    \end{align*}
    By the principles of contradiction, it follows that in case (1), $S$ and $T$ cannot share a spoke outside of $\Omega$.

    Similarly, suppose for contradiction that $W$ exists and is of full index.
    Then 
    \begin{equation*}
        Q := S \cup W \cup T
    \end{equation*}
    is an odd slice of $H$, so:
    \begin{equation*}
        x(Q) \leq t \left( \frac{|Q|-1}{2} \right).
    \end{equation*}
    On the other hand:
    \begin{align*}
        x(Q) &> x(S) + x(T) + x(W) - x(v) \\
        &\geq t \left( \frac{|S| + |T| + |W| - 3}{2} - 1 \right) = t \left( \frac{|Q| - 1}{2}\right).
    \end{align*}
    By the principles of contradiction, in case (2), $W$ is of partial index.
\end{proof}

\begin{figure}
    \centering
\begin{tikzpicture}[scale=1]

\node[fill,circle,draw, scale=0.5] (v1) at (2,2) {};
\node[fill,circle,draw, scale=0.5] (v2) at (2,-2) {};
\node[fill,circle,draw, scale=0.5] (v3) at (-2,-2) {};
\node[fill,circle,draw, scale=0.5] (v4) at (-2,2) {};
\node[fill,circle,draw, scale=0.5] (C) at (0,0) {};

\draw (v1) -- (v2) -- (v3) -- (v4) -- (v1);

\draw (v3) -- (C);

\draw[blue, ultra thick] (v1.237) -- (C.27);
\draw[red,  ultra thick] (v1.205) -- (C.65); 

\draw[blue, ultra thick] (v1) -- (v2) -- (C);
\draw[red, ultra thick] (v1) -- (v4) -- (C);

\path (C) -- node[midway, below right ] {2} (v1);
\path (C) -- node[midway, below left ] {1} (v4);
\path (v1) -- node[midway, above left ] {1} (v4);
\path (C) -- node[midway, below left ] {1} (v2);
\path (v1) -- node[midway, right ] {1} (v2);

\end{tikzpicture}
    \caption{Two adjacent, troublesome slices: A 4-matching on $W_5$ containing two full-index troublesome slices, $S$ (in {\color{red}{red}}) and $T$ (in {\color{blue}{blue}}), that are somewhat, but not almost disjoint.
    Notice that our algorithm can choose either the outer edge of $S$ or the outer edge of $T$, but not both.}
    \label{fig:two adjacent-troublesome slices}
\end{figure}
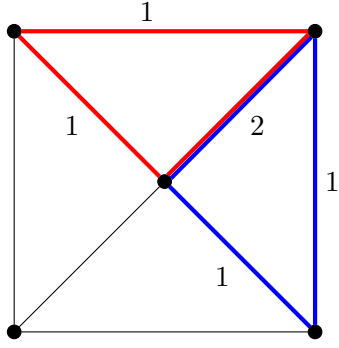

\begin{corollary} \label{cor: Full Index and a Conductor}
    Suppose $H$ contains two full index troublesome slices $S$ and $T$ linked by a (pseudo)conductor $\Omega$.
    Then $x$ splits.
\end{corollary}

\begin{proof}
    By Lemma \ref{lem: two full index and a conductor}, either \emph{(1)} $H$ contains no troublesome slices other than $S$ and $T$ or \emph{(2)} $H$ contains exactly one other troublesome slice $W$. 
    In case (1), $S$ and $T$ do not share a spoke outside of $\Omega$.
    In case (2), $W$ is somewhat but not almost disjoint from $S$ and $T$, and $W$ is only of partial index.

    Suppose $H$ contains no troublesome slices other than $S$ and $T$.
    Then $S$ and $T$ do not share a spoke outside of $\Omega$.
    Select the spoke on the side of $S$ not adjacent to $\Omega$, then select every other edge (outwards from this spoke) along $O(S)$, $\Omega$, and $O(T)$;
    add all of these spokes to $m$.
    Then $m$ assigns sufficient weight to $S$, $T$, and every tight vertex in $\Omega$.
    Some tight vertices may live outside $S$, $T$, and $\Omega$, but since these vertices necessarily live within an insulator, it is easy to find edges to add to $m$ that saturate these vertices (see Subsection \ref{subsec: Insulators} for more details.)

    Now suppose $H$ instead contains $W$.
    Then $W$ is not of full index. 
    We may therefore:
    \begin{enumerate}[(i)]
        \item Treat $S$, $\Omega$, and $T$ as in the previous paragraph,
        \item Find the collection of alternating edges along $O(W)$ that does not saturate any vertex in $S$ or $T$, then
        \item Add this collection to $m$.
    \end{enumerate}
    This process ensures $m$ assigns sufficient weight to $W$. Thus, in every case, we can find a matching $m$ such that $x - m$ is a $(t-1)$-matching, so $x$ splits, as desired.
\end{proof}

Our construction of $m$ is nearly complete.
However, we still need to decide how to choose edges along insulators, and we need to make sure we assign sufficient weight to $v$ and $\Delta$ when they are tight.
We treat this remaining problems in the following subsections:

\subsection{Insulators} \label{subsec: Insulators}
In general, after constructing $m$ along conductors and troublesome slices, we can complete the matching by choosing alternating edges along insulators.
If an insulator is of odd length, this alternating pattern extends without disruption onto the slices to the left and right of the insulator.
Otherwise, there exists some vertex in the insulator that is not tight.
In this case, we can skip the two edges incident to that vertex, so that the alternating pattern to the left of the vertex need not match the alternating pattern to the right, each of which can therefore extend without disruption onto the odd cycles to which they are adjacent.

\begin{figure}
    \centering
    
\begin{tikzpicture}[scale=1.5]

\node (a) at (-3, 0)
{
    \begin{tikzpicture}[scale=1.5]

\foreach \x/\y/\name in {1/2/v1, 2/3/v2, 3/4/v3, 4/5/v4, 5/6/v5, 6/1/v6}
    \node[fill,circle,draw, scale=0.2] (\name) at ({\x*60}:1.5) {$\name$};

\node[fill,circle,draw, scale=0.2] (C) at (0,0) {$C$};

\foreach \a/\b in {v1/v2, v2/v3, v3/v4, v4/v5, v5/v6, v6/v1}
    \draw (\a) -- (\b);

\draw[ultra thick, black] (v3) -- (v4);
\draw[ultra thick, orange] (v3) -- (v2);
\draw[ultra thick, black] (v2) -- (v1);
\draw[ultra thick, orange] (v1) -- (v6);
\draw[ultra thick, black] (v6) -- (v5);
\draw[ultra thick, red] (v6) -- (C) --(v5);
\draw[ultra thick, blue] (v4) -- (C) --(v3);

\path (C) -- node[midway, above ] {1} (v3);
\path (v2) -- node[midway, above left] {1} (v3);
\path (v2) -- node[midway, above] {3} (v1);
\path (v6) -- node[midway, above right] {1} (v1);
\path (v6) -- node[midway,  right] {2} (v5);
\path (v6) -- node[midway,  above] {1} (C);
\path (v5) -- node[midway,  below left] {1} (C);
\path (v4) -- node[midway,  left] {1} (C);
\path (v4) -- node[midway,  left] {2} (v3);

\foreach \vertex in {v1, v2, v3, v4, v5, v6}
    \draw (C) -- (\vertex);

\end{tikzpicture}
};
\node (b) at (2,0)
{
      \begin{tikzpicture}[scale=1.5]

\foreach \x/\y/\name in {1/2/v1, 2/3/v2, 3/4/v3, 4/5/v4, 5/6/v5, 6/1/v6, 3.5/4.5/v7}
    \node[fill,circle,draw, scale=0.2]  (\name) at ({\x*60}:1.5) {$\name$};

\node[fill,circle,draw, scale=0.2]  (C) at (0,0) {$C$};

\node[circle,draw, scale=0.8]  (star) at ({3*60}:1.5) {$\star$};

\foreach \a/\b in {v1/v2, v2/v3, v3/v7, v4/v5, v5/v6, v6/v1, v4/v7}
    \draw (\a) -- (\b);

\draw[ultra thick, orange] (v3) -- (v2);
\draw[ultra thick, black] (v2) -- (v1);
\draw[ultra thick, orange] (v1) -- (v6);
\draw[ultra thick, black] (v6) -- (v5);
\draw[ultra thick, red] (v6) -- (C) --(v5);
\draw[ultra thick, blue] (v4) -- (C) --(v7);
\draw[ultra thick, black] (v4) -- (v7);
\draw[ultra thick, orange] (v3) -- (v7);

\path (C) -- node[midway, above ] {1} (v7);
\path (v2) -- node[midway, above left] {1} (v3);
\path (v2) -- node[midway, above] {3} (v1);
\path (v6) -- node[midway, above right] {1} (v1);
\path (v6) -- node[midway,  right] {2} (v5);
\path (v6) -- node[midway,  above] {1} (C);
\path (v5) -- node[midway,  below left] {1} (C);
\path (v4) -- node[midway,  left] {1} (C);
\path (v4) -- node[midway,  below left] {2} (v7);
\path (v7) -- node[midway,   left] {1} (v3);

\foreach \vertex in {v1, v2, v3, v4, v5, v6, v7}
    \draw (C) -- (\vertex);

\end{tikzpicture}
};
\end{tikzpicture}
\caption{Left: Extending the matchings suggested by two full-index troublesome slices, $S$ (in {\color{red}{red}}) and $T$ (in {\color{blue}{blue}}) in the $4$-matching, shown across an even (vertex) length insulator.\\
Right: Extending along an odd length insulator. Notice that the circled far-left vertex is \emph{loose}. }
label{fig:propogation along insulators}
\end{figure}
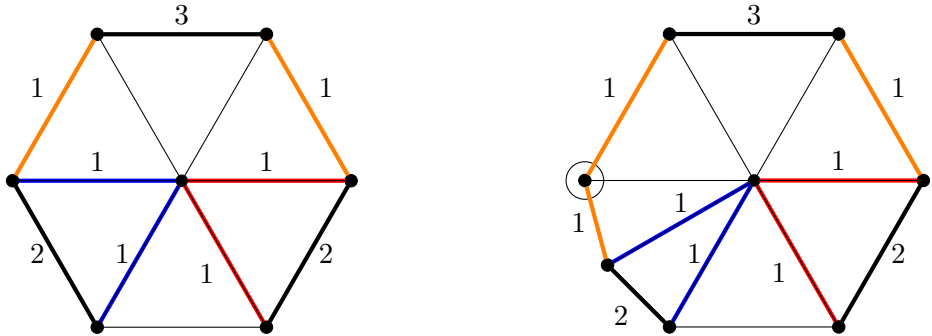

There is one potential complication: we must skip two edges of $W_n$ in a row if the second is not in $H$.
This causes a problem if the vertex between these two edges is tight and not already a part of some troublesome slice. 
In this case, treat the first edge and the spoke attached to the tight vertex as two edges of a triangle.
We can treat this triangle as a tight odd slice as this does not disrupt any of the proofs given up to this point.
In this way, the complication can be addressed using tools we have already developed. 

\subsection{Choosing a Spoke}

When two full-index troublesome slices are joined by a (pseudo)conductor, $m$ already assigns sufficient weight to $v$.
However, when $v$ is tight and this structure does not appear, we may need to modify our construction.

Suppose first that $H$ contains some troublesome slice $S$. 
If $H$ contains any full-index troublesome slices, then take $S$ to be one of those slices.
We can choose $\frac{|S|-1}{2}$ edges of $S$ by choosing one of the outermost spokes of $S$ then proceeding in an alternating pattern along $O(S)$.
Both vertices on the boundary of $O(S)$ are saturated by this choice, but this cannot cause a problem because $S$ is, by assumption, not linked to a full-index troublesome slice by a conductor.

Now, suppose $H$ contains no troublesome slices, but that $v$ is tight. 
Let $e$ be a spoke of $H$, and let $u$ denote the vertex of $\Delta$ incident to $e$.
The algorithm may choose $e$, then proceed along the edges of $\Delta$ in an alternating pattern. 
This causes a problem if and only if every vertex in $\Delta \setminus \{ u \}$ is tight.
In that case, $u$ cannot be tight because:
\begin{equation*}
    x(H) \leq t \left( \frac{|H|-1}{2} \right)
\end{equation*}
ensures that at least one vertex must always be loose.
Thus, $e$ cannot have weight $t$, so there is some other spoke $f$ of $H$.
We may choose $f$, then proceed in an alternating pattern around $\Delta$ without issue.

\subsection{The Outer Cycle}

Thus far, we have ignored the complications posed by $\Delta$ when $n$ is even. 
If $H$ is missing any of the outer edges of $W_n$, then $\Delta$ is not $2$-connected in $H$, so it need not concern us. 
Suppose, therefore, that $H$ contains all the outer edges of $W_n$.
We would like to select the edges of $\Delta$ in an alternating pattern, so that $m(\Delta) = \frac{|\Delta|-1}{2}$. 
However, the alternating pattern is broken whenever insulators resolve the discrepancy between two full-index troublesome slices (as per Corollary \ref{cor: Full Index and a Conductor}). 

Fortunately, we need never deal with this pathology. 
Suppose $\Delta$ is tight, and $S$ is a troublesome slice of $H$ containing the spoke included in $m$ (chosen as in the proof of Corollary \ref{cor: Full Index and a Conductor}). 
Let:
\begin{equation*}
    \xi := \Delta \setminus O(S).
\end{equation*}
Then $|\xi|$ is odd. 
Let $Y$ denote the collective weight of all the spokes in $S$.
Then we have 
\begin{equation*}
    x(S) + x(\xi) - Y = x(\Delta)
\end{equation*}
and
\begin{equation*}
    \frac{|S|-1}{2} + \frac{|\xi|+1}{2} = \frac{|\Delta|+1}{2}.
\end{equation*}

It follows that:
\begin{align*}
    I(\Delta) &= (t-1) \left(\frac{|\Delta|-1}{2}\right) - x(\Delta) \\
    &= (t-1) \left( \frac{|S|-1}{2}+\frac{|\xi| + 1}{2} \right) - x(S) - x(T) + \left(Y - (t-1) \right).
\end{align*}

When $Y \leq t-1$, we have $I(\Delta) \leq I(S) + I(\xi)$, so we need not worry about $\Delta$: it suffices to sufficiently reduce the weight of $S$ and the weight of the tight vertices in $\xi$.
On the other hand, when $Y = t$, $S$ is the only troublesome slice of $H$, so the algorithm does in fact select the edges of $\Delta$ in an alternating pattern. It follows that the outer cycle does not disrupt our construction.

Thus, we have shown that, for every subgraph $H \subseteq W_n$ and nondegenerate $t$-matching $x$ on $H$, we can find a matching $m$ on $H$ such that $x-m$ is a $(t-1)$-matching on $H$.
In short, every nondegenerate $t$-matching on every subgraph of $W_n$ splits.
Thus, by Lemma \ref{lem: Inductive Method} we at last arrive at the following result:

\begin{theorem}
    The matching polytope $P_M(W_n)$ has the IDP.
\end{theorem}

\begin{proof}
    This is a consequence of the work above.
\end{proof}


\section{Conclusion \& Future Work}

As mentioned in the introduction, some of the motivation for our work comes from Ehrhart theory.
Recall that for a lattice polytope $P$, the \emph{Ehrhart series} of $P$ is the generating function:
\[
  \operatorname{Ehr}(P;z):=\sum_{t\in  {\Z_>0}}|tP \cap \Z^n| \, z^t=\frac{h^*(P;z)}{(1-z)^{\operatorname{dim}(P)+1}} \, .
\]
The polynomial $h^*(P;z)=\sum_{i=0}^n h^*_iz^i$ is known as the \emph{$h^*$-polynomial} of~$P$ and is of degree less than $n+1$ and has nonnegative integer coefficients.
The coefficients of $h^*(P;z)$ form the $h^*$-vector of $P$.

In Theorem \ref{thm:gor} we characterized the Gorenstein matching polytopes, and in Corollary \ref{cor:gorenstein_idp} we saw that if $P_M(G)$ is Gorenstein, then it has the IDP. 
Thus, invoking Adiprasito, Papadakis, Petrotou, and Steinmeyer's result (Theorem \ref{AdiprasitoPapadakisPetrotouSteinmeyer}), we obtain the following corollary: 

\begin{corollary}
All Gorenstein matching polytopes have unimodal $h^*$-vectors.    
\end{corollary}

But what about Schepers and Van Langenhoven's question:
If $P$ has the IDP, is it true that $h^*(P;z)$ is unimodal? 

We saw in Section \ref{sec:wheel_graphs_IDP} that the matching polytopes of wheel graphs are not Gorenstein, but still exhibit the IDP. 
While we do not answer Schepers and Van Langenhoven's question here, our work presents directions for fruitful exploration.  
In support of an affirmative answer to their question for matching polytopes of wheel graphs, we present Table \ref{tab:two-column}, which demonstrates that the $h^*$-vector of $P_M(W_n)$ is unimodal for $n=4,\dots, 10$. 

\begin{table}[ht]
    \centering
    \begin{tabular}{|c|c|}
        \hline
        {${n}$} & {$h^*(P_M(W_n))$} \\
        \hline
        \hline
        4 & [1, 3, 3, 1]
 \\
        \hline
        5 & [1, 10, 29, 26, 5]
 \\
        \hline
        6 & [1, 25, 170, 386, 285, 57]
 \\
        \hline
        7 & [1, 53, 714, 3249, 5420, 3262, 567]
 \\
         \hline
        8 & [1, 105, 2590, 21260, 67235, 86451, 43218, 6604]
 \\
        \hline
        9 & [1, 198, 8403, 115516, 624069, 1456544, 1499583, 639382, 87671]
 \\
        \hline
        10 & [1, 363, 25494, 557752, 4805292, 18458122, 33400578, 28421692, 10537374, 1310974]
 \\
        \hline
    \end{tabular}
    \caption{$h^*$-vectors of $P_M(W_n)$ for $n=4, \dots, 10$.}
    \label{tab:two-column}
\end{table}

And this leads us to conjecture the following.

\begin{conjecture}
The $h^*$-vector of the matching polytope of the wheel graph $W_n$, for $n\geq 4$, is unimodal.  
\end{conjecture}

\section*{Acknowledgements \& Funding}
The authors thank Akihiro Higashitani for helpful comments on the paper and for funding Matsushita's visit to UC Berkeley in the Fall of 2023.
The authors also thank Alex Bouquet, Carson Mitchell, and Jupiter Davis, for their interest and ideas at the start of this project. 
We extend a special thank you to J. Carlos Mart\'inez Mori who introduced us to matching polytopes in the first place.

Vindas-Mel\'endez was partially supported by the National Science Foundation under Award DMS-2102921.
This material is also based in part upon work supported by the NSF Grant DMS-1928930 and the Alfred P. Sloan Foundation under grant G-2021-16778, while ARVM was in residence at the Simons Laufer Mathematical Sciences Institute in Berkeley, California, during the Fall 2023 semester.
Matsushita is partially supported by Grant-in-Aid for JSPS Fellows Grant JP22J20033.
The authors also thank the UC Berkeley Mathematics Department for funding their computational resources.

\bibliographystyle{amsplain}
\bibliography{bib.bib}

\end{document}